\documentclass[12pt]{article}
\usepackage{amsmath,amssymb,amsthm,mathrsfs,amscd,amsfonts,latexsym}
\usepackage[cp1251]{inputenc}
\usepackage[ukrainian,english]{babel}

\begin{document}
\pagestyle{myheadings}
\thispagestyle{empty}
\setcounter{page}{1}
\newtheorem{definition}{Definition}
\newtheorem{proposition}{Proposition}
\newtheorem{theorem}{Theorem}
\newtheorem{lemma}{Lemma}
\newtheorem{corollary}{Corollary}
\newtheorem{remark}{Remark}

\newcommand{\Aut}{\mbox{Aut}}

\theoremstyle{plain}
\mathsurround 2pt

\textbf{Local nearrings on finite non-abelian $2$-generated $p$-groups}

\

\textbf{Iryna Iu. Raievska}\\
\emph{Institute of Mathematics of National Academy of Sciences of Ukraine,\\
3, Tereshchenkivs'ka Str., Kyiv, Ukraine, 01024}\\
raeirina@imath.kiev.ua

\

\textbf{Maryna Iu. Raievska}\\
\emph{Institute of Mathematics of National Academy of Sciences of Ukraine,\\
3, Tereshchenkivs'ka Str., Kyiv, Ukraine, 01024}\\
raemarina@imath.kiev.ua

\begin{abstract}
Finite non-abelian non-metacyclic $2$-generated $p$-groups (${p>2}$) of nilpotency class $2$ with cyclic commutator subgroup which are the additive groups of local nearrings are described. It is shown that the subgroup of all non-invertible elements is of index $p$ in its additive group.
\end{abstract}



\emph{An abbreviated title:} Local nearrings\\

\emph{This work was partially supported by the grant 346300 for IMPAN from the Simons Foundation and the matching 2015-2019 Polish MNiSW fund.}

\section{Introduction}\label{sec:intro}

Nearrings are a generalization of associative rings in the sense that with the respect to addition they need not be commutative and only one distributive law is assumed. In this paper the concept ``nearring'' means a left distributive nearring with a multiplicative identity. The reader is referred to the books by Meldrum~\cite{MJ_85} or Pilz~\cite{GP_77} for terminology, definitions and basic facts concerning nearrings.

Following~\cite{MCJ_68}, the nearring with identity will be called local, if the set of all non-invertible elements forms a subgroup of its additive group. The main results concerning local nearrings are summarized in~\cite{S_08}.

In~\cite{CM_70} it is shown that every non-cyclic abelian $p$-group of order $p^n>4$ is the additive group of a zero-symmetric local nearring which is not a ring. As it was noted in~\cite{CM_71}, neither a generalized quaternion group nor a non-abelian group of order $8$ can be the additive group of a local nearring.

Therefore the structure of the non-abelian finite $p$-groups which are the additive groups of local nearrings is an open problem~\cite{Fei_06}.

It was proved that every non-metacyclic Miller--Moreno $p$-group of order $p^n>8$ is the additive group of a local nearring and the multiplicative group of such a nearring is the group of order $p^{n-1}(p-1)$ \cite{IMYa_12}. In this paper finite non-abelian non-metacyclic $2$-generated $p$-groups ($p>2$) of nilpotency class $2$ with cyclic commutator subgroup  are studied.

\section{Preliminaries}

Let $G$ be a finite non-abelian non-metacyclic $2$-generated $p$-group ($p>2$) of nilpotency class $2$ with cyclic commutator subgroup.

Denote by $G'$ and $Z(G)$ the commutator subgroup and the centre of $G$, respectively.

Let $a$ and $b$ be generators for $G$ such that $G/G'=\langle aG'\rangle \times \langle bG'\rangle$, $aG'$ has order $p^m$ and $bG'$ has order $p^n$. Then $c=[a,b]$ generates $G'$, $c$ has order $p^d$ with $1\leq d\leq n\leq m$, and $c\in Z(G)=\langle a^{p^m}, b^{p^n}, c\rangle$.

Suppose that $\langle a \rangle \cap G'=\langle b \rangle \cap G'=1$. Then
\[
G=\langle a,~b,~c| a^{p^m}=b^{p^n}=c^{p^d}=1,~a^b=ac,~c^a=c^b=c\rangle
\]
and each element of $G$ can be uniquely written in the form $a^{x_1}b^{x_2}c^{x_3}$, $x_1\in C_{p^m}$, $x_2\in C_{p^n}$, $x_3\in C_{p^d}$. Therefore the group $G$ with $p > 2$ will be denoted by $G(p^m,p^n,p^d)$.

\begin{lemma}\label{lemma_3}
For any natural numbers $k$ and $l$ the equality $[a^k,b^l]=c^{kl}$ holds.
\end{lemma}

\begin{proof}
Since $b^{-1}ab=ac$, it follows that $b^{-l}ab^l=ac^l$. Therefore, $b^{-l}a^kb^l=(ac^l)^k=a^kc^{kl}$, thus $a^{-k}b^{-l}a^kb^l=c^{kl}$.
\end{proof}

\begin{corollary}\label{corollary_1}
Let the group $G(p^m,p^n,p^d)$ be additively written. Then for any natural numbers $k$ and $l$ the equalities $-ak-bl+ak+bl=c(kl)$ and $bl+ak=-c(kl)+ak+bl$ hold.
\end{corollary}

\begin{lemma}\label{lemma_4}
For any natural numbers $k$, $l$ and $r$ the equality
\[
(a^kb^l)^r=a^{kr}b^{lr}c^{-kl\binom{r}{2}} \leqno ({\rm *})
\]
holds.
\end{lemma}

\begin{proof}
For $r=1$, there is nothing to prove. By induction on $r$, we derive
\[
(a^kb^l)^r=a^{kr}b^{lr}c^{-kl\binom{r}{2}}.
\]
Replacing $r$ by $r+1$ in equality $(*)$, we have
\[
(a^kb^l)^{(r+1)}=a^{kr}b^{lr}a^kb^lc^{-kl\binom{r}{2}}=a^{k(r+1)}b^{l(r+1)}c^{-klr}c^{-kl\binom{r}{2}}
\]
\[
=a^{k(r+1)}b^{l(r+1)}c^{-kl(r+\binom{r}{2})}=a^{k(r+1)}b^{l(r+1)}c^{kl\binom{r+1}{2}}.
\]
Thus, equality $(*)$ holds for an arbitrary $r$.
\end{proof}

\begin{corollary}\label{corollary_2}
Let the group $G(p^m,p^n,p^d)$ be additively written. Then for any natural numbers $k$, $l$ and $r$ the equality $(ak+bl)r=akr+blr-ckl\binom{r}{2}$ holds.
\end{corollary}

Obviously, the exponent of $G(p^m,p^n,p^d)$ is equal to $p^m$ for ${1\leq d\leq n\leq m}$.

\begin{lemma}\label{lemma_6}
If $x$ is an element of order $p^m$ of $G(p^m,p^n,p^d)$, then there exist generators $a$, $b$, $c$ of this group such that $a=x$ and ${a^{p^m}=b^{p^n}=c^{p^d}=1}$, $a^b=ac$, $c^a=c^b=c$.
\end{lemma}

\begin{proof}

Indeed, for each $x\in G(p^m,p^n,p^d)$ there exist positive integers $\alpha$, $\beta$ and $\gamma$ such that $x=a^\alpha b^\beta c^\gamma$. Thus, we have
\[
x^{p^m}=(a^\alpha b^\beta c^\gamma)^{p^m}=(a^\alpha b^\beta)^{p^m} c^{\gamma {p^m}}=a^{\alpha p^m}b^{\beta p^m}c^{\gamma {p^m} -\alpha \beta {p^m\choose 2}}
\]
\[
={a^{p^m}}^\alpha b^{{p^m}\beta} c^{p^m(\gamma-\alpha \beta \frac {(p^m-1)}{2})}=1
\]
by Lemma~\ref{lemma_4}. Since $|a|=p^m$ and $1\leq d\leq n \leq m$, where $m>1$ and $p>2$, it follows that the exponent of $G(p^m,p^n,p^d)$ equals $p^m$.

If
\[
x^{p^{m-1}}=a^{p^{m-1}\alpha} b^{p^{m-1}\beta}c^{p^{m-1}(\gamma-\alpha \beta \frac {(p^{m-1}-1)}{2})}\neq 1,
\]
then either $(\alpha,p)=1$, or $(\beta,p)=1$ for $m=n$, or $(\gamma,p)=1$ for $m=n=d$. So, without loss of generality, we can assume that $(\alpha,p)=1$. Then
\[
\langle x, b \rangle=\langle a^\alpha b^\beta c^\gamma, b \rangle=\langle a^\alpha, b \rangle=\langle a, b \rangle=G
\]
and
\[
b^{-1}xb=b^{-1}(a^\alpha b^\beta c^\gamma)b=(ac)^\alpha b^\beta c^\gamma=(a^\alpha b^\beta c^\gamma)c^\alpha=xc^\alpha.
\]
Furthermore, substituting $c^\alpha$ instead of $c$ for generators $x$ and $b$ of $G(p^m,p^n,p^d)$, we have similar expressions as for generators $a$ and $b$, thus replacing the element $a$ by $x$.
\end{proof}

The following assertion concerning the automorphisms group of $G(p^m,p^n,p^d)$ is a direct consequence of statement~(B1)~\cite{Me_1993}.

\begin{lemma}\label{lemma_5}
Let $G=G(p^m,p^n,p^d)$ and let $\mbox{Aut}(G)$ be the automorphism group of $G$. Then the following statements hold{\rm :}
\begin{itemize}
  \item[\rm{1)}] if $m = n$, then $|\mbox{Aut}(G)|=p^{2d+4m-5}(p^2-1)(p-1)${\rm ;}
  \item[\rm{2)}] if $m > n$, then $|\mbox{Aut}(G)|=p^{2d+3n+m-2}(p-1)^2$.
\end{itemize}
\end{lemma}

An information about a group of automorphisms of $G(p^m,p^m,p^d)$ is given by the following lemma.

\begin{lemma}\label{Aut2}
Let $G=G(p^m,p^m,p^d)$ and let there exist a subgroup $A$ of $\mbox{Aut} G$ of order $p^{2m+d-2}(p^2-1)$, where $m,~d>1$ with odd $p$. If an element $g\in G$ of order $p^m$ and $A$ contains Sylow normal $p$-subgroup, then $G\ne g^A\cup \Phi(G)$.
\end{lemma}
\begin{proof}
Assume that $G=g^A\cup \Phi(G)$. Then $G={(\langle a \rangle\times \langle c \rangle) \rtimes \langle b \rangle}$ with generators $a$, $b$ of order $p^m$ and a central commutator $c=[a,b]$ of order $p^d$ by the definition. Hence
\[
\Phi(G)=(\langle a^p \rangle\times \langle c \rangle)\rtimes \langle b^p \rangle,
\]
and thus all elements of order $p^m$ are contained in $g^A$. Furthermore, $a=g^u$ for some $u\in A$, hence $g^A=a^A$, i.~e. $G=a^A\cup \Phi(G)$. Since $|G|=p^{2m+d}$ and $\Phi(G)=p^{2m+d-2}$, it follows that
\[
|a^A|=|G|-|\Phi(G)|=p^{2m+d-2}(p^2-1),
\]
and so the centralizer $C_A(a)$ of $a$ in $A$ equals $1$. In particular, $(a\langle c^p \rangle)^A=(a\langle c^p \rangle)^B=a\langle c^p \rangle$ for the normal subgroup $B= C_A(a\langle c^p \rangle)$ of order $p^{d-1}$ in $A$.

Considering the factor-group $\bar{G}=G/\langle c^p \rangle$ and $\bar{A}=A/B$. Taking into consideration, that $|\bar{a}^{\bar{A}}|=p^{2m-1}(p^2-1)$, we have $\bar{G}=\bar{a}^{\bar{A}}\cup \Phi(\bar{G})$. Since $|\Phi(\bar{G})|=|Z(\bar{G})|$ and $xy=yx$ for all $x\in \Phi(\bar{G})$, $y\in \bar{G}$, we have $\Phi(\bar{G})=Z(\bar{G})$. Therefore, $\bar{G}$ is a Miller--Moreno group. Since $\bar{G}=\bar{a}^{\bar{A}}\cup Z(\bar{G})$, the latter equality is impossible by~\cite[Lemma 7]{IMYa_12}. This contradiction completes the proof.
\end{proof}

\section{Nearrings with identity on $G(p^m,p^n,p^d)$}\label{sec:fm}

First recall some basic concepts of the theory of nearrings.

\begin{definition}\label{nr}
A set $R$ with two binary operations ``$+"$ and ``$\cdot"$ is called a (left) nearring if the following statements hold:
\begin{enumerate}
  \item [(1)] $(R,+)=R^{+}$ is a (not necessarily abelian) group with neutral element~$0$;
  \item [(2)] $(R,\cdot)$ is a semigroup;
  \item [(3)] $x(y+z)=xy+xz$ for all $x$, $y$, $z\in R$.
\end{enumerate}
\end{definition}

If $R$ is a nearring, then the group $R^+$  is called the {\em additive group} of $R$.  If in addition $0\cdot x=0$, then the nearring $R$ is called {\em zero-symmetric} and if the semigroup $(R,\cdot)$ is a monoid, i.e. it has an identity element $i$, then  $R$ is a {\em nearring with identity} $i$. In the latter case the group $R^*$ of all invertible elements of the monoid $(R,\cdot)$ is called the {\em multiplicative group} of $R$.

The following assertion is well-known.

\begin{lemma}\label{lemma_2}
Let $R$ be a finite nearring with identity $i$. Then the exponent of $R^+$ is equal to the additive order of $i$ which coincides with additive order of every element of $R^*$.
\end{lemma}

As a direct consequence of Lemmas~\ref{lemma_6}~and~\ref{lemma_2} we have the following corollary.

\begin{corollary}\label{corollary_3}
Let $R$ be a nearring with identity $i$ whose group $R^+$ is isomorphic to a group $G(p^m,p^n,p^d)$. Then ${R^{+}=\langle a \rangle+\langle b \rangle+\langle c \rangle}$ with elements $a$, $b$ and $c$, satisfying relations $ap^m=bp^n=cp^d=0$, $-b+a+b=a+c$ and $-a+c+a=-b+c+b=c$ with $1\leq d\leq n \leq m$, where $a=i$.
\end{corollary}

The following statement \cite[Lemma~1]{RS_2012} establishes the connection between the automorphism group of the additive group of the nearring with identity and its multiplicative group.

\begin{lemma}\label{rem}
Let $R$ be a nearring with identity $i$. Then there exists a subgroup $A$ of the automorphism group $\mbox{Aut} R^+$ which is isomorphic to $R^*$ and satisfying the condition $i^A=\{i^a\mid a\in A\}=R^*$.
\end{lemma}

The subgroup $A$ defined in Lemma~\ref{rem} is called the automorphism group of the group $R^+$ associated with the group $R^*$.

The following statement \cite[Theorem~54]{S_08} concerns the structure of $L$ which is the subgroup of all non-invertible elements of finite local nearring $R$. Let $\Phi(G)$ denote the Frattini subgroup of $G$.

\begin{theorem}\label{theorem_FR}
Let $R$ be a local nearring of order $p^n$ and let $G(R) = R^+ \rtimes R^*$ be a group associated with $R$. Then $H=R^+ \rtimes (i + L)$ is a Sylow normal $p$-subgroup of $G(R)$ and $L = R^+ \cap \Phi(H)$. In particular, if $L$ is non-abelian, then its center is non-cyclic.
\end{theorem}

Considering $\Phi(R^+)\leq \Phi(H)$, we have the following corollary.

\begin{corollary}\label{corollary_FR}
$\Phi(R^+)\leq L=\Phi(H)\cap R^+$.
\end{corollary}

Let $R$ be a nearring with identity $i$ whose group $R^+$ is isomorphic to a group $G(p^m,p^n,p^d)$. Thus it follows from Corollary \ref{corollary_3} that ${R^{+}=\langle a \rangle+\langle b \rangle+\langle c \rangle}$ with elements $a$, $b$ and $c$, satisfying relations $ap^m=bp^n=cp^d=0$, $-b+a+b=a+c$ and $-a+c+a=-b+c+b=c$ with $1\leq d\leq n \leq m$, where $a=i$ and each element $x\in R$ is uniquely written in the form ${x=ax_1+bx_2+cx_3}$ with coefficients $0\le x_1<\! p^m$, $0\le x_2<\! p^n$ and $0\le x_3<\! p^d$.

Furthermore, we can assume $xa=ax=x$ for each $x\in R$ and there exist uniquely defined mappings ${\alpha\colon R\to \mathbb Z_{p^m}}$, ${\beta\colon R\to \mathbb Z_{p^n}}$ and ${\gamma\colon R\to \mathbb Z_{p^d}}$ such that
\[
xb=a\alpha(x)+b\beta(x)+c\gamma(x).\leqno (**)
\]

\begin{lemma}\label{lemma_8}
If $x=ax_1+bx_2+cx_3$ and $y=ay_1+by_2+cy_3$ are arbitrary elements of $R$, then
\[
xy=a(x_1y_1+y_2\alpha(x))+b(x_2y_1+y_2\beta(x))+c(-x_1x_2\binom{y_1}{2}
\]
\[
-\binom{y_2}{2}\alpha(x)\beta(x)-x_2y_1y_2\alpha(x)+x_3y_1+y_2\gamma(x)+x_1y_3\beta(x)-x_2y_3\alpha(x)),
\]
where mappings $\alpha\colon R\to \mathbb Z_{p^m}$, $\beta\colon R\to \mathbb Z_{p^n}$ and $\gamma\colon R\to \mathbb Z_{p^d}$ satisfy the conditions{\rm :}
\begin{itemize}
  \item[\rm{(0)}] $\alpha(0)\equiv 0\; (\!\!\mod p^m)$, $\beta(0)\equiv 0\; (\!\!\mod p^n)$ and ${\gamma(0)\equiv 0\; (\!\!\mod p^d)}$ if and only if the nearring $R$ is zero-symmetric{\rm ;}
  \item[\rm{(1)}] $\alpha(xy)\equiv x_1\alpha(y)+\alpha(x)\beta(y)\; (\!\!\mod p^m\;)${\rm ;}
  \item[\rm{(2)}] $\beta(xy)\equiv x_2\alpha(y)+\beta(x)\beta(y)\; (\!\!\mod p^n\;)${\rm ;}
  \item[\rm{(3)}] $\gamma(xy)\equiv -x_1x_2\binom{\alpha(y)}{2}-\alpha(x)\beta(x)\binom{\beta(y)}{2}-x_2\alpha(x)\alpha(y)\beta(y)$
\[
+x_3\alpha(y)+\gamma(x)\beta(y)+x_1\beta(x)\gamma(y)-x_2\alpha(x)\gamma(y)\; (\!\!\mod p^d\;).
\]
\end{itemize}
\end{lemma}
\begin{proof}
If  $R$ is a zero-symmetric nearring, then
\[
0=0\cdot b=a\alpha(0)+b\beta(0)+c\gamma(0),
\]
thus $\alpha(0)\equiv 0\; (\!\!\mod p^m)$, ${\beta(0)\equiv 0\; (\!\!\mod p^n)}$ and ${\gamma(0)\equiv 0\; (\!\!\mod p^d)}$. On the other hand, if the last congruences hold, then $0\cdot b=a\cdot 0+b\cdot 0+c\cdot 0=0$. Since $a$ is the multiplicative identity in $R$, we have $0\cdot a=a\cdot 0=0$. Moreover, from the equality $c=-a-b+a+b$ and the left distributive law it follows that $0\cdot c=-0\cdot a-0\cdot b+0\cdot a+0\cdot b=0$, hence
\[
0\cdot x=0\cdot (ax_1+bx_2+cx_3)=(0\cdot a)x_1+(0\cdot b)x_2+(0\cdot c)x_3=0.
\]
This proves statement~(0).

Next, using~(**) and Corollary~\ref{corollary_1}, we obtain
\[
xc=-xa-xb+xa+xb=-cx_3-bx_2-ax_1-c\gamma(x)-b\beta(x)-a\alpha(x)
\]
\[
+ax_1+bx_2+cx_3+a\alpha(x)+b\beta(x)+c\gamma(x)
\]
\[
=-bx_2-ax_1-b\beta(x)-a\alpha(x)+ax_1+bx_2+a\alpha(x)+b\beta(x)
\]
\[
=-bx_2+cx_1\beta(x)-b\beta(x)-ax_1-a(\alpha(x)-x_1)+bx_2+a\alpha(x)+b\beta(x)
\]
\[
=cx_1\beta(x)-b(x_2+\beta(x))-a\alpha(x)+bx_2+a\alpha(x)+b\beta(x)
\]
\[
=cx_1\beta(x)-b(x_2+\beta(x))-a\alpha(x)-cx_2\alpha(x)+a\alpha(x)+bx_2+b\beta(x)
\]
\[
=c(x_1\beta(x)-x_2\alpha(x))-b(x_2+\beta(x))+bx_2+b\beta(x)=c(x_1\beta(x)-x_2\alpha(x)).
\]
Therefore
\[
xy=(ax_1+bx_2+cx_3)y_1+(a\alpha(x)+b\beta(x)+c\gamma(x))y_2
\]
\[
+(cx_1\beta(x)-x_2\alpha(x))y_3.
\]
Corollary~\ref{corollary_2} implies that
\[
(ax_1+bx_2)y_1=ax_1y_1+bx_2y_1-cx_1x_2\binom{y_1}{2},
\]
\[
(a\alpha(x)+b\beta(x))y_2=ay_2\alpha(x)+by_2\beta(x)-c\binom{y_2}{2}\alpha(x)\beta(x)
\]
and
\[
bx_2y_1+ay_2\alpha(x)=ay_2\alpha(x)+bx_2y_1-cx_2y_1y_2\alpha(x).
\]
By the left distributive law, we have
\[
xy=a(x_1y_1+y_2\alpha(x))+b(x_2y_1+y_2\beta(x))+c(-x_1x_2\binom{y_1}{2}
\]
\[
-\binom{y_2}{2}\alpha(x)\beta(x)-x_2y_1y_2\alpha(x)+x_3y_1+y_2\gamma(x)+x_1y_3\beta(x)-x_2y_3\alpha(x)).
\]
Finally, the associativity of multiplication for all $x$, $y\in R$ implies that
\[
(xy)b=x(yb).\leqno 1)
\]

Thus
\[
(xy)b=a\alpha(xy)+b\beta(xy)+c\gamma(xy)\leqno 2)
\]
 and $yb=a\alpha(y)+b\beta(y)+c\gamma(y)$ by formula (**). Substituting the last expression in the right part of equality~1), we get
\[
x(yb)=a(x_1\alpha(y)+\alpha(x)\beta(y))+b(x_2\alpha(y)+\beta(x)\beta(y))\leqno 3)
\]
\[
+c(-x_1x_2\binom{\alpha(y)}{2}-\alpha(x)\beta(x)\binom{\beta(y)}{2}-x_2\alpha(x)\alpha(y)\beta(y)
\]
\[
+x_3\alpha(y)+\gamma(x)\beta(y)+x_1\beta(x)\gamma(y)-x_2\alpha(x)\gamma(y)).
\]
Comparing the coefficients $a$, $b$ and $c$ in~2)~and~3) by equality~1), we derive statements~(1)--(3) of the lemma.
\end{proof}

\section{Local nearrings on $G(p^m,p^n,p^d)$}

Let $R$ be a local nearring with identity $i$ whose group $R^+$ is isomorphic to the group $G(p^m,p^n,p^d)$. Then ${R^{+}=\langle a \rangle+\langle b \rangle+\langle c \rangle}$ with elements $a$, $b$ and $c$, satisfying relations $ap^m=bp^n=cp^d=0$, $-b+a+b=a+c$ and $-a+c+a=-b+c+b=c$ with $1\leq d\leq n \leq m$, where $a=i$ and each element $x\in R$ is uniquely written in the form ${x=ax_1+bx_2+cx_3}$ with coefficients $0\le x_1<\! p^m$, $0\le x_2<\! p^n$ and $0\le x_3<\! p^d$.

We show that the set $L$ of all non-invertible elements of $R$ is a subgroup of index $p$ in $R^+$.

\begin{theorem}\label{theorem_1}
The following statements hold{\rm :}
\begin{enumerate}
   \item ${L=\langle a\cdot p \rangle+\langle b \rangle+\langle c \rangle}$ and, in particular, the subgroup $L$ is of index $p$ in $R^+$ and $|R^*|=p^{m+n+d-1}(p-1)${\rm ;}
   \item $x=ax_1+bx_2+cx_3$ is an invertible element if and only if ${x_1\not\equiv 0 \; (\!\!\mod p\;)}$.
 \end{enumerate}
\end{theorem}
\begin{proof}
Assume that $|R^+:L|=p^t$, $t>1$. Since $R=R^*\cup L$, it follows that
\[
|R^*|=|R|-|L|=p^{m+n+d}-p^{m+n+d-t}=p^{m+n+d-t}(p^t-1).
\]
According to Lemma~\ref{rem}, the group $R^*$ is isomorphic to the subgroup $A$ of the automorphism group of $R^+$ and so $|R^*|$ divides $|\mbox{Aut} R^+|$. According to statement~1) of Lemma~\ref{lemma_5} it is possible only if $t=2$ and $m=n$.

Assume that $|R^+:L|=p^2$ and $m=n$. If $d=1$, then it is impossible because of~\cite[Theorem 2]{IMYa_12}. Now let $d>1$. Since $|R^+:\Phi(R^+)|=p^2$ and Corollary~\ref{corollary_FR}, we have $L=\Phi(R^+)$. Hence by Lemma~\ref{rem}, we get $R^+=a^A\cup \Phi(R^+)$, which is impossible by Lemma~\ref{Aut2}. This contradiction shows that our assumption is false and so $|R^+:L|=p$.

It is clear that $R/L$ is a nearfield and so the factor-group $R^+/L^+$ is an elementary abelian $p$-group. Thus for $a\notin L$ we have $ap\in L$ and so ${L=\langle a\cdot p \rangle+\langle b \rangle+\langle c \rangle}$.
Therefore $R^*=R\setminus L$ and hence
\[
{R^*=\{ax_1+bx_2+cx_3\mid x_1\not\equiv 0 \; (\!\!\mod p\;)\}}.
\]
\end{proof}

Applying statement~(1) of Theorem~\ref{theorem_1} to Lemma~\ref{lemma_8}, we get the following formula for multiplying elements $x=ax_1+bx_2+cx_3$ and $y=ay_1+by_2+cy_3$ in the local nearring $R$.

\begin{corollary}\label{corollary_4}
If $x$, $y\in R$ with $1\leq d\leq n \leq m$ and $xb=a\alpha(x)+b\beta(x)+c\gamma(x)$, then
\[
xy=a(x_1y_1+y_2\alpha(x))+b(x_2y_1+y_2\beta(x))+c(-x_1x_2\binom{y_1}{2}
\]
\[
-\binom{y_2}{2}\alpha(x)\beta(x)-x_2y_1y_2\alpha(x)+x_3y_1+y_2\gamma(x)+x_1y_3\beta(x)-x_2y_3\alpha(x)),
\]
where mappings $\alpha\colon R\to \mathbb Z_{p^m}$, $\beta\colon R\to \mathbb Z_{p^n}$ and $\gamma\colon R\to \mathbb Z_{p^d}$ and the following statements hold{\rm :}
\begin{itemize}
 \item[\rm{(0)}] $\alpha(0)\equiv 0\; (\!\!\mod p^m)$, $\beta(0)\equiv 0\; (\!\!\mod p^n)$ and ${\gamma(0)\equiv 0\; (\!\!\mod p^d)}$ if and only if the nearring $R$ is zero-symmetric{\rm ;}
 \item[\rm{(1)}] $\alpha(x)\equiv 0 \; (\!\!\mod p)${\rm ;}
 \item[\rm{(2)}] if $\beta(x)\equiv 0 \; (\!\!\mod p)$, then $x_1\equiv 0 \; (\!\!\mod p)${\rm ;}
 \item[\rm{(3)}] $\alpha(xy)\equiv x_1\alpha(y)+\alpha(x)\beta(y)\; (\!\!\mod p^m\;)${\rm ;}
 \item[\rm{(4)}] $\beta(xy)\equiv x_2\alpha(y)+\beta(x)\beta(y)\; (\!\!\mod p^n\;)${\rm ;}
 \item[\rm{(5)}] $\gamma(xy)\equiv -x_1x_2\binom{\alpha(y)}{2}-\alpha(x)\beta(x)\binom{\beta(y)}{2}-x_2\alpha(x)\alpha(y)\beta(y)$
\[
+x_3\alpha(y)+\gamma(x)\beta(y)+x_1\beta(x)\gamma(y)-x_2\alpha(x)\gamma(y)\; (\!\!\mod p^d\;).
\]
\end{itemize}
\end{corollary}
\begin{proof}
Indeed, statements~(0), (3)--(5) repeat statements~(0)--(4) of Lemma~\ref{lemma_8}. Since $L=\langle a\cdot p \rangle+\langle b \rangle+\langle c \rangle$ by Theorem~\ref{theorem_1} and $L$ is an $(R,R)$-subgroup in $R$ by statement~2)~\cite[Lemma~3.2]{AHS_2004}, it follows that $xb\in L$ and hence $\alpha(x)\equiv 0 \; (\!\!\mod p)$, proving statement~(1). Taking $y=c$, we have $xc=c(x_1\beta(x)-x_2\alpha(x))$. Thus, if $\beta(x)\equiv 0 \; (\!\!\mod p)$, then $xc=0\; (\!\!\mod p\;)$, and so $x\in L$. Thus $x_1\equiv 0 \; (\!\!\mod p)$ by Theorem~\ref{theorem_1}, proving statement~(2).
\end{proof}

The following theorem shows that conditions given in Theorem~\ref{theorem_1} are sufficient for existing of finite local nearrings on $G(p^m,p^n,p^d)$. Moreover, each group $G(p^m,p^n,p^d)$ is the additive group of a nearring with identity.

\begin{theorem}\label{theorem_2}
For each prime $p$ and positive integers $m$, $n$ and $d$ with $1 \leq d \leq n \leq m$  there exists a local nearring $R$ whose additive group $R^+$ is isomorphic to the group $G(p^m,p^n,p^d)$.
\end{theorem}
\begin{proof}
Let $R$ be an additively written group $G(p^m,p^n,p^d)$ with generators $a$, $b$ and $c$ satisfying the relations $|a|=p^m$, $|b|=p^n$, $|c|=p^d$, $b^{-1}ab=ac$ and $a^{-1}ca=b^{-1}cb=c$. Then ${G=\langle a \rangle+\langle b \rangle+\langle c \rangle}$ and each element $x\in R$ is uniquely written in the form $x=ax_1+bx_2+cx_3$ with coefficients $0\le x_1<\! p^m$, $0\le x_2<\! p^n$ and $0\le x_3<\! p^d$. In order to define a multiplication ``$\cdot$'' on $R$ in such a manner that $(R,+,\cdot)$ is a local nearring.

Assume that $1 \leq d \leq n \leq m$ and let the mappings from Corollary~\ref{corollary_4} be defined by the congruences ${\alpha(x)\equiv 0\; (\!\!\mod p^m)}$, $\beta(x)\equiv x_1\; (\!\!\mod p^n)$ and $\gamma(x)\equiv 0\; (\!\!\mod p^d)$ for each $x\in G$. Then
\[
x\cdot y=ax_1y_1+b(x_2y_1+x_1y_2)+c(-x_1x_2\binom{y_1}{2}+x_3y_1+x_1^2y_3).
\]

It suffices to show that the mappings $\alpha: G\to \mathbb Z_{p^m}$, $\beta: G\to \mathbb Z_{p^n}$ and $\gamma: G\to \mathbb Z_{p^d}$ with respect to the multiplication ``$\cdot$'' satisfy statements~(0)--(5) of Corollary~\ref{corollary_4}.

Indeed, $\alpha(0)\equiv 0\; (\!\!\mod p^m)$, $\beta(0)\equiv 0\; (\!\!\mod p^n)$ and ${\gamma(0)\equiv 0\; (\!\!\mod p^d)}$
by the definition. Since ${0\cdot y=a\cdot 0+b\cdot 0+c\cdot 0=0}$ for each $y\in G$, this implies that a multiplication ``$\cdot$'' is zero-symmetric and so, proving statement~(0) of Corollary~\ref{corollary_4}. Indeed, we have  $\alpha(x)\equiv 0\; (\!\!\mod p)$ and ${x_1\equiv 0 \; (\!\!\mod p)}$, if ${\beta(x)\equiv 0 \; (\!\!\mod p)}$, so that statements~(1)~and~(2) of Corollary~\ref{corollary_4} hold. Clearly, we derive statements~(3)--(5) of Corollary~\ref{corollary_4}.
\end{proof}

As corollary we have the following assertion.

\begin{corollary}\label{corollary_5}
Each group $G(p^m,p^n,p^d)$ is the additive group of a nearring with identity.
\end{corollary}

\end{document}